\documentclass[11pt]{article}

\usepackage[T1]{fontenc}
\usepackage[utf8]{inputenc}
\usepackage{lmodern}

\usepackage{graphicx}
\usepackage{amsthm, amsmath, amssymb, tikz, bm}
\usepackage{mathtools, intcalc}
\usepackage{xifthen}
\usepackage[colorlinks=true,
linkcolor=blue,citecolor=blue,
urlcolor=blue]{hyperref}
\usepackage{subcaption}

\usetikzlibrary{arrows}
\usepackage{thmtools}
\usepackage{thm-restate}

\usepackage[margin=1in]{geometry}

\usepackage[shortlabels]{enumitem}
\usepackage{todonotes}
\usetikzlibrary{calc,shapes,backgrounds,positioning,fit}
\allowdisplaybreaks

\newtheorem{theorem}{Theorem}
\newtheorem{lemma}[theorem]{Lemma}

\newtheorem{conjecture}{Conjecture}
\newtheorem{proposition}[theorem]{Proposition}

\theoremstyle{definition}

\newcommand{\E}{\operatorname{E}}

\newcommand{\cT}{\mathcal{T}}

\usepackage[normalem]{ulem}

\pgfdeclarelayer{nodelayer}
\pgfdeclarelayer{edgelayer}
\pgfsetlayers{edgelayer,nodelayer,main}
\tikzstyle{Black Vertex}=[fill=black, draw=black, shape=circle]
\tikzstyle{none}=[fill=none, draw=none, shape=circle]
\tikzstyle{Green Edge}=[-, fill=none, draw=green]
\tikzstyle{Red Edge}=[-, fill=none, draw=red]
\tikzstyle{Blue Edge}=[-, draw=blue]
\tikzstyle{Orange Edge}=[-, draw={rgb,255: red,255; green,128; blue,0}]
\tikzstyle{Outline}=[fill=white, draw=black, shape=rectangle]
\tikzstyle{Black Edge}=[-]

\title{Toughness Bounds for Fractional Hamiltonicity and Resistance Positivity}

\author{
Zhiyu Wang \thanks{Louisiana State University, Baton Rouge, LA, 70803
({\tt zhiyuw@lsu.edu}). This author was supported in part by LA Board of Regents grant LEQSF(2024-27)-RD-A-16.}
}
\date{}
\def\e{\mathbf e}

\DeclareMathOperator{\opt}{OPT}
\DeclareMathOperator{\supp}{supp}

\begin{document}

\maketitle

\begin{abstract}
A graph is fractionally Hamiltonian if it admits a nonnegative edge weighting in $[0,1]$
of total weight equal to its order such that every nontrivial edge cut has
weight at least two. Motivated by Chv\'atal's Toughness Conjecture,
Scheinerman and Ullman conjectured that every $2$-tough graph is
fractionally Hamiltonian. In this paper, we show that every connected graph on at least three vertices that is not fractionally
Hamiltonian has a non-Hamiltonian chordal spanning supergraph. Since adding edges does not decrease
toughness, a theorem of Kabela and Kaiser that every $10$-tough chordal
graph on at least three vertices is Hamiltonian yields that every $10$-tough
graph on at least three vertices is fractionally Hamiltonian.

We apply this result to resistance curvature. We prove that every
fractionally Hamiltonian graph is resistance positive (RP), and consequently
every $10$-tough graph is RP, confirming a conjecture of Devriendt. In the
other direction, for every $\varepsilon>0$, we construct a graph that is not
resistance nonnegative and has toughness greater than
$3/2-\varepsilon$, extending a recent construction of Agrahari, Bibby, Boros, Garcia, Heidercheidt, and Wang.
\end{abstract}

\vspace{-1.5em}

\section{Introduction}\label{sec:intro}

Throughout the paper, all graphs are finite and simple. Unless stated otherwise, they are connected. A graph is \emph{Hamiltonian} if it contains a cycle through all of its vertices. For a graph $G$ and a vertex set $S\subseteq V(G)$, let $c(G-S)$ denote the number of connected components of
$G-S$. Given a noncomplete graph $G$, the \emph{toughness} of $G$, denoted by $\tau(G)$, is 
\[
   \tau(G):=\min\left\{\frac{|S|}{c(G-S)}:
   S\subseteq V(G),\ c(G-S)>1\right\}.
\]
Complete graphs are defined to have infinite toughness, and $G$ is called \emph{$t$-tough}
if $\tau(G)\ge t$. Toughness was introduced by Chv\'atal as a necessary
condition for Hamiltonicity~\cite{Chvatal1973}. Every Hamiltonian graph is
$1$-tough, and Chv\'atal made the following conjecture.

\begin{conjecture}[Chv\'atal~\cite{Chvatal1973}]
\label{conj:Chvatal}
There exists a constant $t_0$ such that every $t_0$-tough graph on at least
three vertices is Hamiltonian.
\end{conjecture}

Despite extensive research, Conjecture~\ref{conj:Chvatal} remains open; see
the surveys of Bauer, Broersma, and Schmeichel
\cite{BauerBroersmaSchmeichel2006} and Broersma~\cite{Broersma2015}. 
Enomoto, Jackson, Katerinis, and Saito proved that every $2$-tough graph on
at least three vertices has a $2$-factor, while for every $\varepsilon>0$ there exists a
$(2-\varepsilon)$-tough graph without a $2$-factor
\cite{EnomotoJacksonKaterinisSaito1985}. Bauer, Broersma, and Veldman
constructed, for every $\varepsilon>0$, a $(9/4-\varepsilon)$-tough graph
with no Hamiltonian path~\cite{BauerBroersmaVeldman2000}. Consequently, any
constant $t_0$ in Conjecture~\ref{conj:Chvatal} must satisfy $t_0\ge9/4$.
For other results concerning $2$-factors in tough
graphs, see, e.g.,
\cite{EnomotoJacksonKaterinisSaito1985,Shan2025Planar, Shan2026TriangleFree} and the references therein.
Substantial progress has also been made for restricted graph classes and
under additional degree conditions; see, e.g.,
\cite{Keil1985,KratschLehelMuller1996,DeogunKratschSteiner1997,
BauerBroersmavanDenHeuvelVeldman1995,KabelaKaiser2017,
Shan2020, OtaSanka2022,SankaShan2024}.

Scheinerman and Ullman defined fractional Hamiltonicity through the
following linear relaxation of a Hamiltonian
cycle~\cite[Section~2.3]{ScheinermanUllman1997}. For a
nonempty proper set $S\subsetneq V(G)$, let
$\partial_G(S):=\{uv\in E(G):|\{u,v\}\cap S|=1\}$. When the underlying graph
is clear, we write $\partial(S)$. For an edge weighting $x:E(G)\to\mathbb R$, write $x_e:=x(e)$ for
$e\in E(G)$ and, for $F\subseteq E(G)$, write
$x(F):=\sum_{e\in F}x_e$. An edge weighting
$f:E(G)\to[0,1]$ is a \emph{fractional Hamiltonian cycle} if
$f(E(G))=|V(G)|$ and $f(\partial(S))\ge2$ for every nonempty proper set
$S\subsetneq V(G)$. A graph is \emph{fractionally Hamiltonian} if it has a
fractional Hamiltonian cycle.
A similar formulation also appears in combinatorial optimization and was studied by Boyd and Elliott-Magwood under the term
\emph{SEP-feasible}~\cite{BoydElliottMagwood2007}. 

Scheinerman and Ullman proved that every
fractionally Hamiltonian graph is $1$-tough and constructed
non-fractionally Hamiltonian graphs whose toughness approaches $3/2$ from
below~\cite[Theorem~2.3.4 and Proposition~2.3.5]{ScheinermanUllman1997}.
They conjectured that every $2$-tough graph is fractionally Hamiltonian
\cite[Conjecture~2.3.6]{ScheinermanUllman1997}\footnote{Wang~\cite{Wang2009} claimed to prove this conjecture.
However, the proof as written does not establish the required upper bound
on the dual optimum: the displayed dual-feasible solution of value $n$
shows only that $\opt(D)\ge n$, whereas one must prove
$\opt(D)\le n$.}. The non-fractional statement that
every $2$-tough graph is Hamiltonian, which motivated their discussion, was
subsequently disproved by the construction of Bauer, Broersma, and
Veldman~\cite{BauerBroersmaVeldman2000}, though that construction does not
by itself disprove the fractional conjecture.
Define the critical toughness infimum for fractional Hamiltonicity by
\[
 t^*_{\mathrm{FH}}:=\inf\{t>0:\text{ every $t$-tough graph on at least
 three vertices is fractionally Hamiltonian}\}.
\]
The examples of Scheinerman and Ullman give $t^*_{\mathrm{FH}}\ge3/2$.
Our first main result relates the fractional Hamiltonicity of a graph to the
Hamiltonicity of a chordal spanning supergraph. A graph is \emph{chordal} if it has no
induced cycle of length at least four. The \emph{intersection graph} of a
family of sets has one vertex for each set, with two vertices adjacent
precisely when the corresponding sets intersect. Gavril proved that a graph
is chordal if and only if it is the intersection graph of a family of
subtrees of a tree~\cite{Gavril1974}.

\begin{restatable}{theorem}{transferTheorem}
\label{thm:transfer}
Every connected graph on at least three vertices that is not fractionally
Hamiltonian has a non-Hamiltonian chordal spanning supergraph.
\end{restatable}

The proof constructs the required chordal supergraph from an optimal
solution of the dual fractional-Hamiltonian linear program with laminar
support.
Since adding edges does not decrease toughness, Theorem~\ref{thm:transfer} implies
that if every $q$-tough chordal graph on at least three vertices is
Hamiltonian, then every $q$-tough graph on at least three vertices is
fractionally Hamiltonian. Indeed, otherwise a $q$-tough graph that is not
fractionally Hamiltonian would have a non-Hamiltonian chordal spanning
supergraph, which would also be $q$-tough. Chen, Jacobson, K\'ezdy, and Lehel~\cite{ChenJacobsonKezdyLehel1998}
proved that every $18$-tough chordal graph on at least three vertices is Hamiltonian. Kabela and
Kaiser~\cite{KabelaKaiser2017} improved this bound to $10$. Therefore, Theorem~\ref{thm:transfer}, together with Kabela and
Kaiser's result, immediately gives the following consequence.

\begin{restatable}{theorem}{tentoughFractional}
\label{thm:10-tough-fractional}
Every $10$-tough graph on at least three vertices is fractionally
Hamiltonian.
\end{restatable}

A recent preprint of Huang~\cite{Huang2026} states that every
$5$-tough chordal graph is Hamilton-connected.
Assuming Huang's result, Theorem~\ref{thm:transfer} would
strengthen Theorem~\ref{thm:10-tough-fractional} by replacing $10$ with
$5$. However, in this paper we make no attempt to optimize this constant.

We next discuss the application to resistance curvature. Ricci curvature
plays an important role in the geometric analysis of Riemannian manifolds,
and several analogues have been introduced for nonsmooth and discrete spaces,
including Bakry--\'{E}mery curvature~\cite{BakryEmery1985}, Ollivier's
coarse Ricci curvature~\cite{Ollivier2009}, the synthetic lower curvature
bounds of Lott--Villani and Sturm
\cite{LottVillani2009,Sturm2006I,Sturm2006II}, Forman curvature
\cite{Forman2003}, and Lin-Lu-Yau curvature~\cite{LinLuYau2011}; see also
\cite{ChungYau1996, Higuchi2001, LinYau2010, BauerChungLinLiu2017} and the
references therein.

Devriendt and Lambiotte introduced a vertex curvature based on effective
resistance~\cite{DevriendtLambiotte2022}. Effective resistance originates in
electrical network theory and has a rich probabilistic and geometric theory;
see, for example,~\cite{KleinRandic1993,LyonsPeres2016}. Further properties
of resistance curvature and resistance distance were studied in
\cite{Devriendt2022,DevriendtOttoliniSteinerberger2024,Devriendt2026}.
Let $G=(V,E)$ be a finite simple connected graph, let
$c=(c_e)_{e\in E}$ be positive edge weights, and let $\mathcal T(G)$ be
the set of spanning trees of $G$. For an edge $e=uv$, the
\emph{effective resistance} between $u$ and $v$ is
\[
 \omega_e(c)=c_e^{-1}
 \frac{\displaystyle\sum_{\substack{T\in\mathcal T(G)\\e\in T}}
                 \prod_{h\in T}c_h}
      {\displaystyle\sum_{T\in\mathcal T(G)}\prod_{h\in T}c_h}.
\]
The \emph{relative resistance} of $e$ is $r_e(c):=c_e\omega_e(c)$.
Equivalently, $r_e(c)$ is the probability that $e$ belongs to a random
spanning tree sampled from the log-linear distribution $\mu_c(T)=\frac{\prod_{e\in T}c_e}{\sum_{T'\in\mathcal T(G)}\prod_{e\in T'}c_e}$. Foster's theorem~\cite{Foster1949} gives
$\sum_{e\in E}r_e(c)=|V|-1$. The \emph{resistance curvature} of a vertex
$v$ is $$p_v(c):=1-\frac12\sum_{e\ni v}r_e(c).$$
Following Devriendt~\cite{Devriendt2026}, a graph is \emph{resistance
nonnegative}, abbreviated RN, if it admits positive edge weights $c$ for which
$p_v(c)\ge0$ at every vertex. Similarly, a graph is \emph{resistance positive},
abbreviated RP, if it admits positive edge weights $c$ for which
$p_v(c) > 0$ at every vertex. An RN graph that is not
RP is \emph{strictly resistance nonnegative}, or SRN. A distribution on
$\mathcal T(G)$ is \emph{positive} if every spanning tree receives positive
probability. Devriendt proved the following characterization.

\begin{theorem}\cite[Theorem~3.8]{Devriendt2026}
\label{thm:devriendt-characterization}
A graph $G$ is RN, respectively RP, if and only if there exists a positive
distribution $\mu$ on $\mathcal T(G)$ such that
$\mathbb E_\mu[\deg_T(v)]\le2$, respectively
$\mathbb E_\mu[\deg_T(v)]<2$, for every $v\in V(G)$.
\end{theorem}

Devriendt proved that every Hamiltonian graph is RP
\cite[Theorem~6.2]{Devriendt2026}, while a result of Fiedler
\cite[Theorem~3.4.18]{Fiedler2011}, as applied by Devriendt
\cite[Theorem~1.3]{Devriendt2026}, implies that every RP graph is
$1$-tough. Motivated by these containments, Devriendt made the following
conjecture.

\begin{conjecture}\cite[Conjecture~6.4]{Devriendt2026}
\label{conj:tough-RP}
There is a constant $t_{\mathrm{RP}}$ such that every
$t_{\mathrm{RP}}$-tough graph is RP.
\end{conjecture}

Devriendt~\cite[Question~6.6]{Devriendt2026} asked whether every $1$-tough graph is RP, which was very recently disproved by Agrahari, Bibby, Boros, Garcia, Heidercheidt, and Wang~\cite{AgrahariBibbyBorosGarciaHeidercheidtWang2026}. 
We show the following theorem connecting fractional Hamiltonicity and the RP property. 

\begin{restatable}{theorem}{fractionalHamiltonianRP}
\label{thm:fractional-Hamiltonian-RP}
Every fractionally Hamiltonian graph is RP.
\end{restatable}

Combining Theorem~\ref{thm:fractional-Hamiltonian-RP} with
Theorem~\ref{thm:10-tough-fractional} resolves
Conjecture~\ref{conj:tough-RP} affirmatively.

\begin{restatable}{corollary}{tentoughRP}
\label{thm:10-tough-RP}
Every $10$-tough graph is RP.
\end{restatable}

Finally, we give lower bounds for the toughness needed to force resistance
nonnegativity and resistance positivity. For integers $m\ge2$ and
$\ell\ge1$, let $H_m$ be obtained from the complete graph on
$\{a,b_1,\ldots,b_m\}$ by subdividing each edge $ab_i$ once, with
subdivision vertex $x_i$. This is the one-subdivision case, called $G_m(1,1,\cdots,1)$, of the family
studied by Agrahari, Bibby, Boros, Garcia, Heidercheidt, and Wang
\cite{AgrahariBibbyBorosGarciaHeidercheidtWang2026}. Inspired by their construction, we define 
$J_{m,\ell}:=K_\ell\vee H_m$, where $\vee$ denotes the join operation. We determine the
exact toughness of $J_{m,\ell}$ and show that it is not RN when
$m\ge2\ell+4$.

\begin{restatable}{theorem}{almostThreeHalves}
\label{thm:almost-3/2}
For every $\varepsilon>0$, there exists a graph $G$ that is not RN and
satisfies $\tau(G)>\frac32-\varepsilon$.
\end{restatable}

The preceding results give
\[
 \{\text{Hamiltonian graphs}\}\subsetneq
 \{\text{fractionally Hamiltonian graphs}\}\subsetneq
 \{\text{RP graphs}\}\subsetneq
 \{\text{$1$-tough graphs}\}.
\]
Each containment relation is strict. Indeed, the Petersen graph is fractionally Hamiltonian but not Hamiltonian
\cite[Section~2.3]{ScheinermanUllman1997}; the graph $H_3$ is RP but not
fractionally Hamiltonian, as shown in
Section~\ref{sec:resistance}; and, for every $m\ge4$, the graph $H_m$ is
$1$-tough but not RP by the argument in
\cite[Section~3]{AgrahariBibbyBorosGarciaHeidercheidtWang2026}.
Similarly, define
\[
\begin{aligned}
 t^*_{\mathrm{RN}}&:=\inf\{t>0:\text{ every $t$-tough graph is RN}\},\\
 t^*_{\mathrm{RP}}&:=\inf\{t>0:\text{ every $t$-tough graph is RP}\}.
\end{aligned}
\]
Since fractional Hamiltonicity implies RP and RP implies RN, the thresholds
satisfy $t^*_{\mathrm{RN}}\le t^*_{\mathrm{RP}}\le t^*_{\mathrm{FH}}$.
Our results give
\[
 \frac32\le t^*_{\mathrm{RN}}\le t^*_{\mathrm{RP}}
 \le t^*_{\mathrm{FH}}\le10.
\]

\medskip

{\noindent\bf Organization and Notation.}
In Section~\ref{sec:threshold}, we prove Theorem~\ref{thm:transfer} and
Theorem~\ref{thm:10-tough-fractional}. In
Section~\ref{sec:resistance}, we prove
Theorem~\ref{thm:fractional-Hamiltonian-RP} and
Corollary~\ref{thm:10-tough-RP}. In Section~\ref{sec:construction}, we study
the graphs $J_{m,\ell}$ and prove Theorem~\ref{thm:almost-3/2}. For a spanning tree
$T$, we write $\deg_T(v)$ for the degree of $v$ in $T$, and for a graph $G$
we write $\mathcal T(G)$ for its set of spanning trees. For $v\in V(G)$,
let $E(v)$ denote the set of edges incident with $v$. For a positive integer $k$, write $[k]:=\{1,\ldots,k\}$.

\section{A finite toughness threshold for fractional Hamiltonicity}
\label{sec:threshold}

Throughout this section, let $G=(V,E)$ be a connected graph of order
$n\ge3$. Let
$\mathcal C(G):=\{\partial(S):\emptyset\neq S\subsetneq V\}$ be the
collection of distinct nontrivial edge cuts of $G$. Thus a cut is included
only once, even though $\partial(S)=\partial(V\setminus S)$. Since $G$ is
connected, every $F\in\mathcal C(G)$ determines a unique unordered
bipartition of $V$. We call $A$ and $V\setminus A$ the two \emph{shores}
of the cut $F=\partial(A)$, and say that $F$ \emph{separates} two vertices
$u,v$ if they lie in different shores.

We use the fractional Hamiltonian linear program of Scheinerman and
Ullman~\cite[Section~2.3]{ScheinermanUllman1997}. 
For readers less familiar with linear programming, we recall the form of
duality used here. If $A$ is a finite matrix, then the dual of
$\min\{c^{\mathsf T}x:Ax\ge b,\ x\ge0\}$ is
$\max\{b^{\mathsf T}y:A^{\mathsf T}y\le c,\ y\ge0\}$. Weak duality
states that $b^{\mathsf T}y\le c^{\mathsf T}x$ for every feasible primal
vector $x$ and feasible dual vector $y$. If the primal is feasible and has a
finite optimum, then the dual has the same optimal value by the strong
duality theorem; see, for example,~\cite[Chapter~5]{Schrijver2003}.

Let $A$ be the $\mathcal C(G)\times E$ incidence matrix defined by
$A_{F,e}=1$ if $e\in F$ and $A_{F,e}=0$ otherwise. The primal variable is
$x=(x_e)_{e\in E}\in\mathbb R^E$, and the dual variable is
$y=(y_F)_{F\in\mathcal C(G)}\in\mathbb R^{\mathcal C(G)}$. The primal and
dual programs are
\begin{equation*}
\tag{$P$}\label{LP:P}
\begin{aligned}
  \text{minimize}\quad &x(E)\\
  \text{subject to}\quad
     &x(F)\ge2 &&(F\in\mathcal C(G)),\\
     &x_e\ge0 &&(e\in E),
\end{aligned}
\end{equation*}
and
\begin{equation*}
\tag{$D$}\label{LP:D}
\begin{aligned}
  \text{maximize}\quad &2\sum_{F\in\mathcal C(G)}y_F\\
  \text{subject to}\quad
     &\sum_{\substack{F\in\mathcal C(G)\\e\in F}}y_F\le1
        &&(e\in E),\\
     &y_F\ge0 &&(F\in\mathcal C(G)).
\end{aligned}
\end{equation*}
We write $\opt(P)$ and $\opt(D)$ for their optimal objective values. 
For completeness, weak duality for this particular pair follows directly.
If $x$ is feasible for~\eqref{LP:P} and $y$ is feasible for~\eqref{LP:D},
then
\begin{equation}\label{eq:weak-duality}
  2\sum_{F\in\mathcal C(G)}y_F
  \le \sum_{F\in\mathcal C(G)}y_Fx(F)
  =\sum_{e\in E}x_e
      \sum_{\substack{F\in\mathcal C(G)\\e\in F}}y_F
  \le\sum_{e\in E}x_e.
\end{equation}
Scheinerman and Ullman~\cite[Proposition~2.3.1]{ScheinermanUllman1997}
showed the following equivalence between fractional Hamiltonicity and its LP
formulation. We include its proof for completeness.

\begin{lemma}\label{lem:LP-characterization}~\cite[Proposition~2.3.1]{ScheinermanUllman1997}
Let $G$ be a graph on at least three vertices. Then $G$ is fractionally Hamiltonian if and only if $\opt(P)=n$.
\end{lemma}

\begin{proof}
Every feasible solution $x$ of~\eqref{LP:P} satisfies
$x(\partial(\{v\}))\ge2$ for every $v\in V$. Therefore
\begin{equation}\label{eq:P-lower}
  2x(E)=\sum_{v\in V}x(\partial(\{v\}))\ge2n,
\end{equation}
so $\opt(P)\ge n$. A fractional Hamiltonian cycle is feasible for
~\eqref{LP:P} and has total weight $n$, proving one direction.

Conversely, suppose that $\opt(P)=n$. Since the feasible region is a nonempty polyhedron and the objective has a finite optimum, the optimum is attained. Let $x$ be an
optimal solution. Equality holds in~\eqref{eq:P-lower}. Since each of the
$n$ singleton-cut terms is at least $2$, we have
$x(\partial(\{v\}))=2$ for every $v\in V$. Let $uv\in E$. Since $n\ge3$,
the set $\{u,v\}$ is nonempty and proper, so
$\partial(\{u,v\})\in\mathcal C(G)$. Hence
$2\le x(\partial(\{u,v\}))=x(\partial(\{u\}))+x(\partial(\{v\}))-2x_{uv}
=4-2x_{uv}$, and therefore $x_{uv}\le1$. Moreover, since $x$ is feasible for~\eqref{LP:P}, we have
$x(\partial(S))\ge2$ for every nonempty proper set $S\subsetneq V$. Thus $x\in[0,1]^E$, while
$x(E)=n$ and every nontrivial cut has weight at least $2$. Hence $x$ is a
fractional Hamiltonian cycle.
\end{proof}

\begin{lemma}\label{lem:singleton-dual}
The dual program~\eqref{LP:D} has a feasible solution of value $n$.
\end{lemma}

\begin{proof}
For every $v\in V$, assign weight $1/2$ to the singleton cut
$\partial(\{v\})$, and assign weight $0$ to every other cut. The singleton
cuts are distinct because $G$ is connected and $n\ge3$. Every edge $uv$
lies in exactly the two singleton cuts $\partial(\{u\})$ and
$\partial(\{v\})$, so its total dual load is $1/2+1/2=1$. The resulting
dual vector is feasible, and its objective value is
$2\sum_{v\in V}(1/2)=n$.
\end{proof}

We next uncross an optimal dual solution. The following is a standard dual uncrossing lemma. Related submodular uncrossing ideas appear in
Edmonds~\cite{Edmonds1970} and Lov\'asz~\cite{Lovasz1976}, while the
polyhedral dual form used here was developed by Edmonds and
Giles~\cite{EdmondsGiles1977}; see also
\cite[Section~60.1]{Schrijver2003}. We include the proof for completeness.
Two cuts $F_A=\partial(A)$ and $F_B=\partial(B)$ \emph{cross} if all four
sets $A\cap B$, $A\setminus B$, $B\setminus A$, and
$V\setminus(A\cup B)$ are nonempty. A family of cuts is \emph{laminar} if no two
of its members cross. For a dual vector $y$, let
$\supp(y):=\{F\in\mathcal C(G):y_F>0\}$.

\begin{lemma}\label{lem:uncrossing}~\cite{EdmondsGiles1977,Schrijver2003}
The dual program~\eqref{LP:D} has an optimal solution $y$ for which
$\supp(y)$ is a laminar family of cuts.
\end{lemma}

\begin{proof}
Since $G$ is connected, every cut $F\in\mathcal C(G)$ is nonempty. If
$e\in F$ and $y$ is dual-feasible, then
$0\le y_F\le\sum_{R\in\mathcal C(G):e\in R}y_R\le1$. Hence the dual
feasible region is a nonempty compact polytope, and an optimal solution
exists.

For $F=\partial(S)\in\mathcal C(G)$, define
$\rho(F):=|S|\,|V\setminus S|$. This is well defined because the expression
is unchanged when $S$ is replaced by its complement. Among all optimal dual
solutions, choose $y$ minimizing
$\Phi(y):=\sum_{F\in\mathcal C(G)}\rho(F)y_F$.

Suppose that $F_A,F_B\in\supp(y)$ cross. Choose shores $A$ and $B$ such
that all four regions determined by them are nonempty, and set
$\varepsilon:=\min\{y_{F_A},y_{F_B}\}>0$. Let
$F_{\cap}:=\partial(A\cap B)$ and $F_{\cup}:=\partial(A\cup B)$; both are
nontrivial cuts. Define $y'$ by subtracting $\varepsilon$ from the
coordinates indexed by $F_A$ and $F_B$, adding $\varepsilon$ to the
coordinates indexed by $F_{\cap}$ and $F_{\cup}$, and leaving every other
coordinate unchanged.

For an edge set $R\subseteq E$, let $\e_R\in\mathbb R^E$ denote its
indicator vector, and for disjoint vertex sets $X,Y$, let $E(X,Y)$ be the
set of edges with one endpoint in $X$ and the other in $Y$. The following
identity holds coordinatewise:
\begin{equation}\label{eq:cut-uncrossing-identity}
 \e_{\partial(A)}+\e_{\partial(B)}
 -\e_{\partial(A\cap B)}-\e_{\partial(A\cup B)}
 =2\e_{E(A\setminus B,B\setminus A)}.
\end{equation}
Indeed, one checks the identity by placing the endpoints of an edge in the
four regions determined by $A$ and $B$; only an edge between
$A\setminus B$ and $B\setminus A$ contributes to the difference. In
particular,
$\e_{F_{\cap}}+\e_{F_{\cup}}\le\e_{F_A}+\e_{F_B}$ coordinatewise. Hence
the modification does not increase the load on any edge, so $y'$ is
feasible. Its objective value is unchanged because the sum of its
coordinates is unchanged.

On the other hand, a direct calculation gives
\begin{align*}
 \Phi(y')-\Phi(y)
 &=\varepsilon\bigl(\rho(F_{\cap})+\rho(F_{\cup})
   -\rho(F_A)-\rho(F_B)\bigr)\\
 &=-2\varepsilon|A\setminus B|\,|B\setminus A|<0,
\end{align*}
contradicting the choice of $y$. Therefore $\supp(y)$ is laminar.
\end{proof}

We next represent a laminar family of cuts by a weighted tree. After
identifying each cut with its unordered bipartition, this is the weighted
form of Buneman's tree representation for compatible split systems; see
\cite[pp.~388--390]{Buneman1971}. We give the rooted inclusion-tree
construction explicitly because the placement map $\alpha$ will be used
later. By a \emph{weighted tree} we mean a tree $Q$ together with a positive length
function $\ell_Q:E(Q)\to\mathbb R_{>0}$. For $q,q'\in V(Q)$, the distance
$d_Q(q,q')$ is the sum of the edge lengths on the unique $q$--$q'$ path.

\begin{lemma}\label{lem:tree}
Let $y$ be feasible for~\eqref{LP:D}, and suppose that
$\mathcal F:=\supp(y)$ is a laminar family of cuts. Then there are a
weighted tree $(Q_y,\ell_{Q_y})$ and a map
$\alpha:V(G)\to V(Q_y)$ such that
\begin{equation}\label{eq:tree-distance}
 d_{Q_y}(\alpha(u),\alpha(v))
 =\sum_{\substack{F\in\mathcal F\\F\text{ separates }u\text{ and }v}}y_F
 \qquad(u,v\in V(G)).
\end{equation}
In particular, $d_{Q_y}(\alpha(u),\alpha(v))\le1$ whenever $uv\in E(G)$.
\end{lemma}

\begin{proof}
Fix an arbitrary root vertex $r\in V$. For every $F\in\mathcal F$, let
$S_F$ be the unique shore of $F$ that does not contain $r$, and let
$\mathcal L:=\{S_F:F\in\mathcal F\}$. Then $\mathcal L$ is a laminar set
family in the usual sense. Indeed, if two members $S_F,S_{F'}$ intersect
but neither contains the other, then $S_F\cap S_{F'}$,
$S_F\setminus S_{F'}$, and $S_{F'}\setminus S_F$ are nonempty, while
$V\setminus(S_F\cup S_{F'})$ contains $r$. The cuts $F$ and $F'$ would
therefore cross, a contradiction.

Define $p:\mathcal L\to\mathcal L\cup\{V\}$ by letting $p(S)$ be the
inclusion-minimal member of $\mathcal L\cup\{V\}$ that properly contains
$S$. Such a member exists and is unique because the proper supersets of $S$
in a finite laminar family form a chain. Define
\begin{equation}\label{eq:Q-definition}
 V(Q_y):=\mathcal L\cup\{V\},\qquad
 E(Q_y):=\bigl\{\{S,p(S)\}:S\in\mathcal L\bigr\}.
\end{equation}
Here the symbol $V$ on the right denotes the whole ground set, regarded as a
vertex of $Q_y$. Define the edge-length function
$\ell_{Q_y}:E(Q_y)\to\mathbb R_{>0}$ by
$\ell_{Q_y}(\{S,p(S)\}):=y_{\partial(S)}$ for every $S\in\mathcal L$.
Repeatedly following parent edges strictly increases the corresponding sets
and eventually reaches $V$, so $Q_y$ is connected. It has
$|\mathcal L|+1$ vertices and $|\mathcal L|$ edges, and hence is a tree.
All edge lengths are positive because the corresponding cuts belong to
$\supp(y)$.

For $v\in V(G)$, define $\alpha(v)$ to be the inclusion-minimal member of
$\mathcal L\cup\{V\}$ containing $v$. This is well defined because the
members containing $v$ form a nonempty chain. The map $\alpha$ need not be
injective.

For $S\in\mathcal L$, remove the tree edge
$e_S:=\{S,p(S)\}$. The component of $Q_y-e_S$ containing the node $S$
consists exactly of the members of $\mathcal L$ contained in $S$. It follows
that $\alpha(v)$ lies in this component if and only if $v\in S$. Therefore
$e_S$ lies on the unique $\alpha(u)$--$\alpha(v)$ path exactly when the cut
$\partial(S)$ separates $u$ and $v$. Summing the lengths
$y_{\partial(S)}$ of these path edges proves~\eqref{eq:tree-distance}.

If $uv\in E(G)$, then a cut $F\in\mathcal F$ separates $u$ and $v$ exactly
when $uv\in F$. Thus the right-hand side of~\eqref{eq:tree-distance} is the
dual load on $uv$, which is at most $1$ by~\eqref{LP:D}.
\end{proof}

We remark that the root vertex in the preceding proof is used only to choose one shore of
each cut and thereby write the inclusion tree. The distance in~\eqref{eq:tree-distance} is expressed solely in terms of the cuts and their
weights, and hence is independent of that choice.

Here is a small example of the construction, independent of dual
feasibility. Let $V=\{r,a,b,c,d\}$ and use $r$ as the root vertex. Let
$\mathcal L$ consist of $\{a\}$, $\{b\}$, $\{a,b\}$, $\{c\}$, and
$\{c,d\}$, with respective weights $1/5$, $3/10$, $2/5$, $1/10$, and
$1/2$. The resulting tree is shown in Figure~\ref{fig:laminar-tree}. We
have $\alpha(a)=\{a\}$, $\alpha(b)=\{b\}$, and
$\alpha(c)=\{c\}$. Also, $\alpha(d)=\{c,d\}$ and $\alpha(r)=V$.
For example, the path from $\alpha(a)$ to $\alpha(d)$ has length
$1/5+2/5+1/2=11/10$. This is the sum of the weights of the three cuts
represented by $\{a\}$, $\{a,b\}$, and $\{c,d\}$ that separate $a$
and $d$.

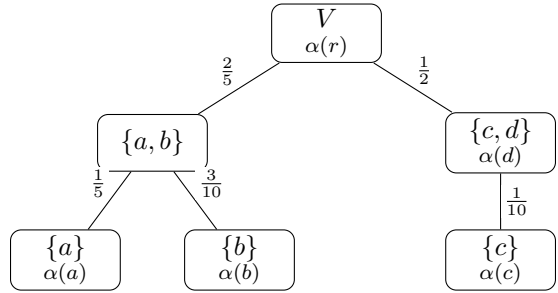
\begin{figure}[ht]
\centering
\begin{tikzpicture}[
  setnode/.style={draw,rounded corners,minimum width=1.45cm,
    minimum height=0.75cm,inner sep=2pt,font=\small,align=center},
  edgelabel/.style={fill=white,draw=none,inner sep=1pt,font=\scriptsize}
]
\node[setnode] (root) at (0,3.0) {$V$\\[-1mm]{\scriptsize $\alpha(r)$}};
\node[setnode] (ab) at (-2.3,1.55) {$\{a,b\}$};
\node[setnode] (cd) at (2.3,1.55) {$\{c,d\}$\\[-1mm]{\scriptsize $\alpha(d)$}};
\node[setnode] (aa) at (-3.45,0) {$\{a\}$\\[-1mm]{\scriptsize $\alpha(a)$}};
\node[setnode] (bb) at (-1.15,0) {$\{b\}$\\[-1mm]{\scriptsize $\alpha(b)$}};
\node[setnode] (cc) at (2.3,0) {$\{c\}$\\[-1mm]{\scriptsize $\alpha(c)$}};

\draw (root)--node[edgelabel,above left] {$\frac25$} (ab);
\draw (root)--node[edgelabel,above right] {$\frac12$} (cd);
\draw (ab)--node[edgelabel,above left] {$\frac15$} (aa);
\draw (ab)--node[edgelabel,above right] {$\frac3{10}$} (bb);
\draw (cd)--node[edgelabel,right] {$\frac1{10}$} (cc);
\end{tikzpicture}
\caption{The weighted inclusion tree associated with an illustrative
laminar family. The vertices of the tree are sets, while the original
vertices are placed by the map $\alpha$.}
\label{fig:laminar-tree}
\end{figure}

Using the weighted tree $Q_y$ and the map $\alpha$, define a graph $H_y$ by
\begin{equation}\label{eq:H-definition}
 V(H_y):=V(G),\qquad
 E(H_y):=\bigl\{uv\in\tbinom{V(G)}2:
 d_{Q_y}(\alpha(u),\alpha(v))\le1\bigr\}.
\end{equation}
Thus $H_y$ is a simple graph on the same vertex set as $G$.

Viewing the geometric realization of $Q_y$ as an $\mathbb R$-tree, the graph $H_y$ is a unit ball graph and hence strongly chordal by Kuroda and Tsujie~\cite[Theorem~1.9]{KurodaTsujie2021}. For completeness,
we give a direct proof of chordality using Gavril's characterization.

\begin{lemma}\label{lem:chordal}
The graph $H_y$ is chordal and contains $G$ as a spanning subgraph.
\end{lemma}

\begin{proof}
If $uv\in E(G)$, then Lemma~\ref{lem:tree} gives
$d_{Q_y}(\alpha(u),\alpha(v))\le1$, so $uv\in E(H_y)$. Hence $G$ is a
spanning subgraph of $H_y$.

We show that $H_y$ is the intersection graph of a family of subtrees of a
finite tree. For each pair of distinct vertices $u,v\in V(G)$ satisfying
$d_{Q_y}(\alpha(u),\alpha(v))\le1$, let $m_{uv}$ be the midpoint of the
unique $\alpha(u)$--$\alpha(v)$ path in $Q_y$, where path length is measured
using the edge lengths of $Q_y$. If $m_{uv}$ lies in the interior of an
edge, subdivide that edge at $m_{uv}$. Performing these subdivisions for all
such pairs produces a finite tree $\widehat Q_y$. Give each new edge the
length of the corresponding portion of the original edge, so that
subdivision does not change path lengths.

For each $v\in V(G)$, let $T_v$ be the subgraph of $\widehat Q_y$ induced
by the vertices $z$ satisfying
\[
 d_{\widehat Q_y}(\alpha(v),z)\le\frac12.
\]
The subgraph $T_v$ is connected. Indeed, if $z\in V(T_v)$, then every
vertex on the unique $\alpha(v)$--$z$ path is no farther from $\alpha(v)$
than $z$ is, and hence also belongs to $T_v$. Thus $T_v$ is a subtree of
$\widehat Q_y$.

We claim that, for distinct $u,v\in V(G)$,
\[
 T_u\cap T_v\neq\emptyset
 \quad\Longleftrightarrow\quad
 d_{Q_y}(\alpha(u),\alpha(v))\le1.
\]
Suppose first that
$d_{Q_y}(\alpha(u),\alpha(v))\le1$. The midpoint $m_{uv}$ is a vertex of
$\widehat Q_y$ and satisfies
\[
 d_{\widehat Q_y}(\alpha(u),m_{uv})
 =d_{\widehat Q_y}(\alpha(v),m_{uv})
 =\frac12d_{Q_y}(\alpha(u),\alpha(v))
 \le\frac12.
\]
Hence $m_{uv}\in V(T_u)\cap V(T_v)$.

Conversely, suppose that $z\in V(T_u)\cap V(T_v)$. Since subdivision does
not change path lengths, the unique paths in $\widehat Q_y$ give
\[
 d_{Q_y}(\alpha(u),\alpha(v))
 \le d_{\widehat Q_y}(\alpha(u),z)
    +d_{\widehat Q_y}(z,\alpha(v))
 \le1.
\]
This proves the claim. By the definition of $H_y$, it follows that $H_y$ is
the intersection graph of the indexed family
$(T_v)_{v\in V(G)}$ of subtrees of the host tree $\widehat Q_y$. Therefore
$H_y$ is chordal by Gavril's theorem~\cite{Gavril1974}.
\end{proof}

We are now ready to prove Theorem~\ref{thm:transfer}, which we restate here
for convenience.

\transferTheorem*

\begin{proof}
Let $G=(V,E)$ be a connected graph, where $n:=|V|\ge3$. The primal
program~\eqref{LP:P} is feasible: assigning $x_e:=2$ for every edge
satisfies every cut constraint. Its objective is bounded below by $n$
by~\eqref{eq:P-lower}. The dual is feasible by taking $y=0$, so strong
duality applies.

By Lemma~\ref{lem:uncrossing}, choose an optimal dual solution $y$ whose
positive support $\mathcal F:=\supp(y)$ is a laminar family of cuts.
Construct the weighted tree $Q_y$, the map $\alpha$, and the graph $H_y$ as
in Lemma~\ref{lem:tree} and~\eqref{eq:H-definition}. By
Lemma~\ref{lem:chordal}, $H_y$ is a chordal spanning supergraph of $G$.

It remains to show that if $H_y$ is Hamiltonian, then $G$ is fractionally
Hamiltonian. Suppose that $H_y$ is Hamiltonian, and let
$C=v_1v_2\cdots v_nv_1$ be a Hamiltonian cycle, where indices are read
modulo $n$. Every edge $v_iv_{i+1}$ of $C$ belongs to $H_y$, and hence
$d_{Q_y}(\alpha(v_i),\alpha(v_{i+1}))\le1$. Summing around the cycle gives
\begin{equation}\label{eq:cycle-upper}
 \sum_{i=1}^{n}d_{Q_y}(\alpha(v_i),\alpha(v_{i+1}))\le n.
\end{equation}
For $F\in\mathcal F$, let $\partial_C(F)$ be the set of edges of $C$ whose
endpoints lie in different shores of $F$. Applying~\eqref{eq:tree-distance}
to every cycle edge and reversing the order of the two finite sums yields
\begin{align}
 \sum_{i=1}^{n}d_{Q_y}(\alpha(v_i),\alpha(v_{i+1}))
 &=\sum_{i=1}^{n}
   \sum_{\substack{F\in\mathcal F\\
   F\text{ separates }v_i\text{ and }v_{i+1}}}y_F \notag\\
 &=\sum_{F\in\mathcal F}y_F|\partial_C(F)|.
 \label{eq:cycle-expand}
\end{align}
Both shores of every $F\in\mathcal F$ are nonempty. While traversing a
Hamiltonian cycle, every passage from one shore to the other must be followed
by a passage back. Hence $|\partial_C(F)|$ is a positive even integer and is
at least $2$. Combining~\eqref{eq:cycle-upper} and
~\eqref{eq:cycle-expand} gives
\[
 n\ge\sum_{F\in\mathcal F}y_F|\partial_C(F)|
   \ge2\sum_{F\in\mathcal F}y_F=\opt(D).
\]
Lemma~\ref{lem:singleton-dual} gives $\opt(D)\ge n$, so $\opt(D)=n$.
Strong duality yields $\opt(P)=n$, and
Lemma~\ref{lem:LP-characterization} implies that $G$ is fractionally
Hamiltonian. Consequently, if $G$ is not fractionally Hamiltonian, then
$H_y$ is a non-Hamiltonian chordal spanning supergraph of $G$.
\end{proof}

We now apply Theorem~\ref{thm:transfer} to the theorem of Kabela and Kaiser.

\tentoughFractional*

\begin{proof}
Let $G$ be a $10$-tough graph on at least three vertices. Suppose, for a
contradiction, that $G$ is not fractionally Hamiltonian. By
Theorem~\ref{thm:transfer}, $G$ has a non-Hamiltonian chordal spanning
supergraph $H$. Since adding edges can only merge components, $H$ is
$10$-tough. Hence $H$ is Hamiltonian by the theorem of Kabela and Kaiser
\cite[Theorem~2]{KabelaKaiser2017}, a contradiction.
\end{proof}

\section{Fractional Hamiltonicity and resistance positivity}
\label{sec:resistance}

For $F\subseteq E(G)$, let $\mathbf e_F\in\mathbb R^{E(G)}$ be its
indicator vector. The \emph{spanning-tree polytope} of $G$ is
$P(G):=\operatorname{conv}\{\mathbf e_T:T\in\mathcal T(G)\}$.
We use the following graphic-matroid case of Edmonds' base-polytope
description; see also~\cite[Chapter~50]{Schrijver2003}.

\begin{theorem}[Edmonds~\cite{Edmonds1970}]
\label{thm:tree-polytope}
Let $G=(V,E)$ be connected. Then $P(G)$ is the set of all
$x=(x_e)_{e\in E}\in\mathbb R^E$ satisfying:
\begin{enumerate}[(1)]
    \item $x_e\ge0$ for every $e\in E$;
    \item $x(E)=|V|-1$;
    \item $x(E[S])\le |S|-1$ for every nonempty proper set
    $S\subsetneq V$, where $E[S]$ is the set of edges with both endpoints
    in $S$.
\end{enumerate}
\end{theorem}

The following lemma immediately follows from Devriendt's
characterization~\cite[Theorem~3.8]{Devriendt2026} and we include its proof here for completeness.

\begin{lemma}\label{lem:strict-marginal-RP}
Let $G=(V,E)$ be connected. If there exists $x\in P(G)$ such that
$x(E(v))<2$ for every $v\in V$, then $G$ is RP.
\end{lemma}

\begin{proof}
Since $x\in P(G)$, there is a distribution $\mu_0$ on $\cT(G)$ whose
edge-marginal vector is $x$. Thus
$\E_{\mu_0}[\deg_T(v)]=x(E(v))<2$ for every $v\in V$.
Let $\mu_1$ be the uniform distribution on $\cT(G)$ and, for
$0<\eta<1$, let $\mu_\eta:=(1-\eta)\mu_0+\eta\mu_1$. For every
sufficiently small $\eta>0$, the distribution $\mu_\eta$ is positive and
$\E_{\mu_\eta}[\deg_T(v)]<2$ for every $v\in V$. The result follows from
Theorem~\ref{thm:devriendt-characterization}.
\end{proof}

We now prove Theorem~\ref{thm:fractional-Hamiltonian-RP}, restated here for
convenience.

\fractionalHamiltonianRP*

\begin{proof}
Let $G=(V,E)$ be an $n$-vertex fractionally Hamiltonian graph, and let $f$
be a fractional Hamiltonian cycle of $G$. The singleton cut inequalities give
$f(E(v))\ge2$ for every $v\in V$. Since
$\sum_{v\in V}f(E(v))=2f(E)=2n$, equality holds at every vertex, so
$f(E(v))=2$ for all $v\in V$.

Set $x:=\frac{n-1}{n}f$. We verify Edmonds' inequalities for the
spanning-tree polytope. Clearly $x_e\ge0$ for every $e\in E$, and
$x(E)=n-1$. 

Let $\emptyset\neq S\subsetneq V$. Counting the weights incident with
the vertices of $S$ gives
\[
 2|S|=\sum_{v\in S}f(E(v))=2f(E[S])+f(\partial(S)).
\]
Hence $f(E[S])=|S|-\frac12f(\partial(S))\le |S|-1$, and therefore
$x(E[S])\le\frac{n-1}{n}(|S|-1)\le |S|-1$. By
Theorem~\ref{thm:tree-polytope}, $x\in P(G)$. Finally, for every $v\in V$,
$x(E(v))=\frac{n-1}{n}f(E(v))=2-\frac2n<2$. Therefore
Lemma~\ref{lem:strict-marginal-RP} implies that $G$ is RP.
\end{proof}

We remark that $H_3$ is RP but not fractionally Hamiltonian. Recall that
$H_m$ is the graph obtained from the complete graph on
$\{a,b_1,\ldots,b_m\}$ by subdividing each edge $ab_i$ once, with
subdivision vertex $x_i$. 
Let $T_0:=\{a x_i,x_i b_i:i\in[3]\}$. Assign probability $2/5$ to $T_0$ and
probability $1/15$ to each of the nine spanning trees
$\{x_i b_i:i\in[3]\}\cup\{a x_k\}\cup R$, where $k\in[3]$ and $R$ is a
spanning tree of the clique on $\{b_1,b_2,b_3\}$. The resulting
edge-marginal vector $x$ satisfies
$x(E(a))=x(E(b_i))=9/5$ and $x(E(x_i))=8/5$ for every $i\in[3]$.
Thus $H_3$ is RP by Lemma~\ref{lem:strict-marginal-RP}. On the other hand,
in any fractional Hamiltonian cycle of $H_3$, the two edges incident with
each degree-two vertex $x_i$ would both have weight $1$, giving total
incident weight $3$ at $a$, contrary to
$f(E(a))=2$.

We now deduce the resistance-curvature consequence of the main theorem.

\tentoughRP*

\begin{proof}
Let $G$ be a $10$-tough graph. If $|V(G)|\le2$, then $G$ is complete and
the result follows directly from
Theorem~\ref{thm:devriendt-characterization}. If $|V(G)|\ge3$, then
Theorem~\ref{thm:10-tough-fractional} shows that $G$ is fractionally
Hamiltonian, and Theorem~\ref{thm:fractional-Hamiltonian-RP} implies that
$G$ is RP.
\end{proof}

\section{Non-RN graphs with toughness approaching
\texorpdfstring{$3/2$}{3/2}}
\label{sec:construction}

We begin with a weighted obstruction that will be used for the graphs
$J_{m,\ell}$ defined in the introduction. The proof is similar to the argument in
\cite[Theorem~4]{AgrahariBibbyBorosGarciaHeidercheidtWang2026}.

\begin{lemma}\label{lem:weighted-obstruction}
Let $G=(V,E)$ be connected, and let
$\lambda:V\to\mathbb R_{\ge0}$ be nonzero. Suppose that
\[
 \sum_{v\in V}\lambda(v)\deg_T(v)\ge
 2\sum_{v\in V}\lambda(v)
\]
for every $T\in\cT(G)$. Then $G$ is
not RP. If this inequality is strict for at least one spanning tree, then
$G$ is not RN.
\end{lemma}

\begin{proof}
Suppose first that $G$ is RP. By
Theorem~\ref{thm:devriendt-characterization}, there is a positive
distribution $\mu$ on $\cT(G)$ such that
$\E_\mu[\deg_T(v)]<2$ for every $v\in V$. Since $\lambda$ is nonzero and
nonnegative,
\[
 \E_\mu\left[\sum_{v\in V}\lambda(v)\deg_T(v)\right]
 =\sum_{v\in V}\lambda(v)\E_\mu[\deg_T(v)]
 <2\sum_{v\in V}\lambda(v),
\]
contradicting the assumed lower bound.

Now suppose that the inequality is strict for some spanning tree $T_1$ and
that $G$ is RN. Let $\mu$ be a positive distribution on $\cT(G)$ such that
$\E_\mu[\deg_T(v)]\le2$ for every $v$. Since $\mu(T_1)>0$, averaging the
assumed inequalities gives
\[
 \E_\mu\left[\sum_{v\in V}\lambda(v)\deg_T(v)\right]
 >2\sum_{v\in V}\lambda(v).
\]
On the other hand,
\[
 \E_\mu\left[\sum_{v\in V}\lambda(v)\deg_T(v)\right]
 =\sum_{v\in V}\lambda(v)\E_\mu[\deg_T(v)]
 \le2\sum_{v\in V}\lambda(v),
\]
a contradiction.
\end{proof}

For an edge $uv\in E(G)$, define
$w_\lambda(uv):=\lambda(u)+\lambda(v)$. Then
$$\sum_{v\in V}\lambda(v)\deg_T(v)
=\sum_{uv\in E(T)}w_\lambda(uv).$$
Thus Lemma~\ref{lem:weighted-obstruction} can be applied by giving a lower
bound on the weight of every spanning tree.

Recall that $H_m=G_m(1,\ldots,1)$ is the one-subdivision case of
the family from~\cite{AgrahariBibbyBorosGarciaHeidercheidtWang2026}: it is
obtained from the complete graph on $\{a,b_1,\ldots,b_m\}$ by subdividing
each edge $ab_i$ once, with subdivision vertex $x_i$. We set
$J_{m,\ell}=K_\ell\vee H_m$. Let $B:=\{b_1,\ldots,b_m\}$ and
$U:=V(K_\ell)$.
Thus $B$ induces a clique, each $a x_i b_i$ is a path, and every vertex of
$U$ is adjacent to every vertex of $H_m$; see
Figure~\ref{fig:Jml}.

\begin{figure}[ht]
\centering
\begin{tikzpicture}[
    scale=1,
    every node/.style={font=\small},
    vtx/.style={circle,draw,fill=white,minimum size=6mm,inner sep=0pt},
    setbox/.style={draw,rounded corners,inner sep=6pt},
    joinline/.style={densely dashed,thick}
]
    \node[vtx] (a) at (0,0) {$a$};

    \node[vtx] (x1) at (2.1,1.45) {$x_1$};
    \node[vtx] (x2) at (2.1,0.45) {$x_2$};
    \node (xdots) at (2.1,-0.45) {$\vdots$};
    \node[vtx] (xm) at (2.1,-1.45) {$x_m$};

    \node[vtx] (b1) at (4.2,1.45) {$b_1$};
    \node[vtx] (b2) at (4.2,0.45) {$b_2$};
    \node (bdots) at (4.2,-0.45) {$\vdots$};
    \node[vtx] (bm) at (4.2,-1.45) {$b_m$};

    \draw (a)--(x1)--(b1);
    \draw (a)--(x2)--(b2);
    \draw (a)--(xm)--(bm);

    \draw (b1)--(b2);
    \draw (b1) to[bend left=18] (bm);
    \draw (b2) to[bend left=10] (bm);
    \node[setbox,fit=(b1)(b2)(bdots)(bm),label=right:$B\cong K_m$] (Bbox) {};

    \node[vtx] (u1) at (-3.0,0.9) {$u_1$};
    \node (udots) at (-3.0,0) {$\vdots$};
    \node[vtx] (ul) at (-3.0,-0.9) {$u_\ell$};
    \draw (u1)--(ul);
    \node[setbox,fit=(u1)(udots)(ul),label=left:$U\cong K_\ell$] (Ubox) {};

    \draw[joinline] (Ubox.east) -- node[above] {complete join} (a.west);
    \draw[joinline] (Ubox.east) -- (x2.west);
    \draw[joinline] (Ubox.east) -- (Bbox.west);
\end{tikzpicture}
\caption{The graph $J_{m,\ell}=K_\ell\vee H_m$. The indexed vertices
indicate the repeating pattern.}
\label{fig:Jml}
\end{figure}
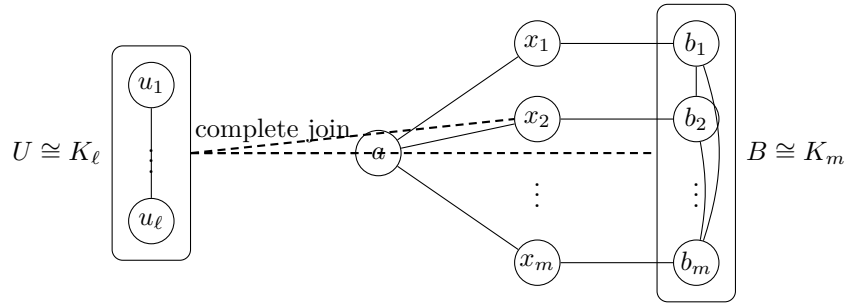

We first compute the toughness of $J_{m,\ell}$.

\begin{proposition}\label{prop:J-toughness}
For all $m\ge2$ and $\ell\ge1$, $\tau(J_{m,\ell})=1+\frac{\ell}{m}$.
\end{proposition}

\begin{proof}
Let $S_0:=U\cup\{a,b_1,\ldots,b_{m-1}\}$. Then $|S_0|=\ell+m$, and
$J_{m,\ell}-S_0$ has the isolated vertices
$x_1,\ldots,x_{m-1}$ and the component induced by $\{x_m,b_m\}$.
Thus $c(J_{m,\ell}-S_0)=m$, and hence
\[
 \tau(J_{m,\ell})\le\frac{\ell+m}{m}
 =1+\frac{\ell}{m}.
\]

For the reverse inequality, let $S\subseteq V(J_{m,\ell})$ satisfy
$c(J_{m,\ell}-S)>1$. If some vertex of $U$ does not belong to $S$, then
it is adjacent to every other vertex of $J_{m,\ell}-S$, a contradiction.
Therefore $U\subseteq S$. Let $R:=S\setminus U$ and
$c:=c(H_m-R)=c(J_{m,\ell}-S)$.

Suppose first that $a\in R$. If $B\setminus R\neq\emptyset$, then
$B\setminus R$ is contained in one component of $H_m-R$, and every other
component is an isolated vertex $x_i$ with $b_i\in R$. Therefore
$c\le1+|R\cap B|\le m$. Moreover, $R$ contains $a$ and a distinct vertex
$b_i$ for each such isolated vertex, so $|R|\ge c$. If $B\subseteq R$,
then $H_m-R$ is an independent set on at most $m$ vertices, while
$|R|\ge m+1$. Thus, in either case, $c\le m$ and $|R|\ge c$. Hence
\[
 \frac{|S|}{c}
 =\frac{\ell+|R|}{c}
 \ge1+\frac{\ell}{c}
 \ge1+\frac{\ell}{m}.
\]

Now suppose that $a\notin R$. Every vertex $x_i\notin R$ lies in the
component containing $a$. If $x_i,b_i\notin R$ for some $i\in[m]$, then
$b_i$, and consequently every vertex of $B\setminus R$, also lies in this
component. It follows that $H_m-R$ is connected, a contradiction. Thus
\[
 \{x_i,b_i\}\cap R\neq\emptyset
 \qquad\text{for every }i\in[m],
\]
and so $|R|\ge m$.

Also, $B\setminus R\neq\emptyset$, since otherwise $H_m-R$ would be
connected. For every $b_i\in B\setminus R$, we have $x_i\in R$.
Consequently, $B\setminus R$ induces one component, and all remaining
vertices of $H_m-R$ lie in the component containing $a$. Thus $c=2$, and
\[
 \frac{|S|}{c}
 =\frac{\ell+|R|}{2}
 \ge\frac{\ell+m}{2}
 \ge\frac{\ell+m}{m}
 =1+\frac{\ell}{m},
\]
where the second inequality follows from $m\ge2$. Together with the upper
bound, this proves the proposition.
\end{proof}

We next show that the graphs in the relevant parameter range are not RN.

\begin{proposition}\label{prop:J-nonRN}
If $m\ge2\ell+4$, then $J_{m,\ell}$ is not RN.
\end{proposition}

\begin{proof}
Define $\lambda:V(J_{m,\ell})\to\mathbb R_{\ge0}$ by
$\lambda(a)=2$, $\lambda(u)=2$ for $u\in U$, $\lambda(b_i)=1$, and
$\lambda(x_i)=0$ for $i\in[m]$. Then
\[
 \sum_{v\in V(J_{m,\ell})}\lambda(v)=m+2\ell+2.
\]
The only edges of weight $1$ under
$w_\lambda(uv)=\lambda(u)+\lambda(v)$ are the $m$ edges $x_i b_i$.
Every other edge has weight at least $2$.

The graph $J_{m,\ell}$ has $2m+\ell+1$ vertices, so every spanning tree has
$2m+\ell$ edges. At most $m$ of these edges have weight $1$. Therefore every
spanning tree $T$ satisfies
\begin{equation}
\begin{aligned}
   \sum_{v\in V(J_{m,\ell})}\lambda(v)\deg_T(v)
   &=\sum_{e\in E(T)}w_\lambda(e)\\
   &\ge m+2\big((2m+\ell)-m\big)\\
   &=3m+2\ell.
\end{aligned}
\end{equation}
The lower bound is attained by the spanning tree $T_0$ consisting of all
edges $x_i b_i$, a spanning tree of $B$, the edge $a x_1$, and the edges
$u x_1$ for $u\in U$.
Since
\[
 (3m+2\ell)
 -2\sum_{v\in V(J_{m,\ell})}\lambda(v)
 =m-2\ell-4,
\]
if $m>2\ell+4$, then every spanning tree $T$ satisfies
\[
 \sum_{v\in V(J_{m,\ell})}\lambda(v)\deg_T(v)
 >2\sum_{v\in V(J_{m,\ell})}\lambda(v).
\]
Lemma~\ref{lem:weighted-obstruction} therefore implies that
$J_{m,\ell}$ is not RN.

It remains to consider $m=2\ell+4$. In this case every spanning tree $T$
satisfies
\[
 \sum_{v\in V(J_{m,\ell})}\lambda(v)\deg_T(v)
 \ge2\sum_{v\in V(J_{m,\ell})}\lambda(v),
\]
and $T_0$ satisfies equality. Choose $u\in U$. In $T_0$, the vertex $u$ is
a leaf incident with $u x_1$. Let $T_1:=T_0-u x_1+u b_1$. Then $T_1$ is a
spanning tree. Since $w_\lambda(u x_1)=2$ and
$w_\lambda(u b_1)=3$, we have
\[
 \sum_{v\in V(J_{m,\ell})}\lambda(v)\deg_{T_1}(v)
 =2\sum_{v\in V(J_{m,\ell})}\lambda(v)+1.
\]
Thus the hypotheses of Lemma~\ref{lem:weighted-obstruction} hold, and
$J_{m,\ell}$ is not RN.
\end{proof}

We now prove Theorem~\ref{thm:almost-3/2}, which we restate here for
convenience.

\almostThreeHalves*

\begin{proof}
For $\ell\ge1$, let $G_\ell:=J_{2\ell+4,\ell}$. By
Proposition~\ref{prop:J-nonRN}, the graph $G_\ell$ is not RN. By
Proposition~\ref{prop:J-toughness},
$\tau(G_\ell)=1+\frac{\ell}{2\ell+4}
=\frac32-\frac1{\ell+2}$.
The result follows by choosing $\ell$ such that $1/(\ell+2)<\varepsilon$.
\end{proof}

\section*{Acknowledgements}
The author used OpenAI's ChatGPT 5.6 Pro extensively in developing this
work. In response to prompts, suggested mathematical formulations, and prior successful approaches and constructions supplied by the author and several rounds of discussions and corrections, ChatGPT 5.6 Pro generated the initial proofs of
Theorems~\ref{thm:transfer} and
\ref{thm:fractional-Hamiltonian-RP}, and proposed the family
$J_{m,\ell}$ used to prove Theorem~\ref{thm:almost-3/2} as an extension of
the construction of Agrahari et
al.~\cite{AgrahariBibbyBorosGarciaHeidercheidtWang2026}. It also assisted
with the organization and exposition of the manuscript. The author
independently verified, corrected, and finalized all arguments and takes full
responsibility for the mathematical content. The author also thanks Xiaonan Liu for helpful discussions.

\end{document}